\documentclass[12pt]{article}
\usepackage{amssymb}
\usepackage{amsmath}
\usepackage{color}
\usepackage{graphicx}
\usepackage{mathrsfs}
\usepackage{listings}

\newcommand{\p}{\mathbf{P}}
\newcommand{\e}{\mathbf{E}}

\newcommand{\F}{{\mathscr{F}}}

\newcommand{\Z}{\mathbb{Z}}
\newcommand{\N}{\mathbb{N}}

\newcommand{\A}{\mathcal{A}}

\newenvironment{proof}{\flushleft{\bf Proof}}

\newcommand{\lb}{\left}
\newcommand{\rb}{\right}

\newcommand{\nn}{\nonumber}

\newcommand{\spa}{\hspace{1cm}}
\newcommand{\given}{\vline\mbox{ }}

\newtheorem{Theorem}{Theorem}[section]
\newtheorem{Lemma}[Theorem]{Lemma}

\newtheorem{Proposition}[Theorem]{Proposition}

\newtheorem{Remark}[Theorem]{Remark}

\newcommand{\ex}[1]{\e\lb[#1\rb]}

\newcommand{\pr}[1]{\p\lb[#1\rb]}

\newcommand{\be}[1]{\begin{equation}\label{#1}}
\newcommand{\ee}{\end{equation}}
\newcommand{\bea}[1]{\begin{eqnarray}\label{#1}}
\newcommand{\eea}{\end{eqnarray}}

\newcommand{\qed}{ \hfill   $\square$ }

\numberwithin{equation}{section}
\begin{document}
\title{Asymptotic behaviour of the critical value for the  contact process with rapid stirring}
\author{Roman Berezin$^1$
\hspace*{3cm}
Leonid Mytnik$^1$\\
\mbox{}\\
\mbox{}\\
Faculty of Industrial Engineering and Management \\ Technion --- Israel Institute of Technology, \\ Haifa 32000, Israel
}
%\address{Faculty of Industrial Engineering and Management Technion Israel Institute of Technology, Haifa 32000, Israel}
%\email{leonid@ie.technion.ac.il}
%\thanks{LM is partly supported by the Israel Science Foundation}

\date{}

\maketitle
\begin{abstract}We study the
behaviour of the contact process
with rapid stirring on the lattice $\Z^d$ in dimensions  $d\geq 3$.
This process was studied earlier by Konno and Katori, who proved
results for the speed of convergence of the critical value
as the rate of stirring approaches infinity. %Katori has improved the result of Katori in $d\geq 3$ and has produced an interval where the critical value must lie.
 In this article we improve  the results of Konno and Katori and establish  the sharp asymptotics of the critical value in dimensions $d\geq 3$.
\end{abstract}

\footnotetext[1]{Supported by the ISF grant 497/10.}
%\footnotetext[2]{Supported by an ISF grant.}


AMS 2000 {\em subject classifications}
Primary 82C22 60K35 ; Secondary 60J80.

%\begin{titlepage}
%\begin{center}
%{\Large \textbf{\\Asymptotic behaviour of the critical value for the  contact process with rapid stirring\\}}
%\\ for $d\geq 3$ \\} }
%\vspace{5cm} {\large
%{Roman Berezin\\ Leonid Mytnik}}\vspace{9cm}

%{\large  \today}
%\end{center}
%\end{titlepage}

%\tableofcontents
\newpage

\section{Introduction and the main result}
In this paper, we are going to study the behaviour of the so-called contact process with rapid stirring
(see \cite{Kat}, \cite{Kon}). The process is also known as
a contact process combined with an exclusion process (see \cite{Lgt1,Lgt2}).
\medskip
\newline DeMasi, Ferrari and Lebowitz (see \cite{MFL}) studied interacting
particle systems on a lattice under the combined influence of spin
flip and simple exchange dynamics. They proved that when the
stirrings occur on a fast time scale of order $\varepsilon^{-2}$ the
macroscopic density, defined on spatial scale $\varepsilon^{-1}$,
evolves according to an autonomous nonlinear diffusion-reaction
equation. Using the
connection between a convergent sequence of such particle systems to
a solution of a reaction-diffusion equation, found by DeMasi,
Ferrari and Lebowitz, Durrett and Neuhauser (in \cite{DN}) proved results about the
existence of phase transitions when the stirring rate is large that
apply to many different systems.
\medskip
\newline
The contact process with rapid stirring  was
studied by Konno (see \cite{Kon}), who
described it via a
system of interacting particles, on a lattice
$\Z^d$ (while for the proofs the process on the rescaled lattice $\Z^d/\sqrt{N}$ was considered).
The state of the process at time $t$ is given by a function $\xi_t^N:\Z^d\rightarrow\{0, 1\}$, where the value of $\xi_t^N(x)$ is determined by the number of particles present at $x$
at time $t$.
In this setting particles die at rate $1$,
and give birth, onto the closest neighbouring sites, at rate
$\lambda$. In addition, %particles
values of $\xi^N_t$ at two neighbouring sites are
exchanged at rate $N$ (stirring rate), and all the above mechanisms are independent.
%Konno's
The primary goal of \cite{Kon} was to improve the result of Durrett and
Neuhauser, who showed the following:

\begin{Theorem}[Durrett, Neuhauser, \cite{DN}]\label{DNTheorem} Let $\{\xi_{t}^N\}_{t\geq 0}$ be the set-valued contact process with stirring, with
the dynamics as described above, starting with a single occupied
site at the origin. Let
$\Omega_\infty=\{\xi_{t}^N\not=\emptyset\mbox{ }\forall t\geq0\}$,
$\rho_\lambda^N=\pr{\Omega_\infty}$, and let
\[\lambda_c(N)=\inf\{\lambda\geq0:\rho_\lambda^N>0\}.\]\index{$\lambda_c$}
As $N\rightarrow \infty$,
\begin{description}
  \item[a.] $\lambda_c(N)\rightarrow 1.$
  \item[b.] If $\lambda>1$, then $\rho_\lambda^N\rightarrow
  (\lambda-1)/\lambda$.
\end{description}
\end{Theorem}

Konno used the methods of \cite{BDS} to get a more detailed picture
of the critical value $\lambda_c(N)$ as the stirring rate $N$
approaches infinity. The main result of~\cite{Kon} can be stated as follows.

\begin{Theorem}[Konno, \cite{Kon}]\label{KTheorem} For all $x>0$, let
\[
\varphi_d(x)=\left\{
  \begin{array}{ll}
    1/x^{1/3}, & d=1, \\
    \log{x}/x, & d=2, \\
    1/x, & d\geq3.
  \end{array}
\right.
\] Then we have \[\lambda_c(N)-1\approx C_\star\varphi_d(N)\]
where $\approx$ means that if $C_\star$ is small (large) then the
right-hand side of the above is a lower (upper) bound of the
left-hand side for large enough $N$.
\end{Theorem}

Moreover, %Konno has refined the results of
Theorem \ref{KTheorem} was refined in
dimensions  $d\geq 3$:
it was shown in~\cite{Kat} that  \bea{konno1}
\frac{1}{(2d)(2d-1)}&\leq&\liminf_{N\rightarrow\infty}N\lb(\lambda_c(N)-1\rb)\\\nn&\leq&\limsup_{N\rightarrow\infty}N\lb(\lambda_c(N)-1\rb)\leq\frac{G(0,0)-1}{2d},
\eea

where $G(\cdot,\cdot)$ is the Green's function for the simple random
walk on $\Z^d$. \newline

The main goal of this article is to show that, in fact,
the lower bound in (\ref{konno1}) can be improved to the value  $\frac{G(0,0)-1}{2d}$. By this, we get the sharp asymptotics of the critical value $\lambda_c(N)$.
Before we state our main result, we need another piece of notation. Let $\eth_d(x)=\{y\in\Z^d:\|y-x\|_1=1\}$ denote "the
neighbourhood of $x$ in $\Z^d$, excluding $x$". Then $\eth_d\equiv\eth_d(0)$ will denote the neighbourhood of the origin.\newline

The main result of this article is
\begin{Theorem}\label{C3} Let $d\geq 3$. Then
\[\lambda_c(N)-1\sim \frac{\vartheta}{N}, \mbox{ as }
N\rightarrow\infty,\] where
\be{vartheta}\vartheta=\frac{1}{4d^2}\sum_{n=1}^\infty\pr{V_n\in\eth_d},\ee
 $\{V_n\}_{n\geq 0}$ is a symmetric random walk on $\Z^d$ starting at the origin, and $\sim$ means the ratio approaches $1$, as $N$ approaches $\infty$.
\end{Theorem}

To connect the result in the theorem with (\ref{konno1}) let us state a simple lemma.
\begin{Lemma}\label{thetaG}
\[ 2d\vartheta=G(0,0)-1.\]
\end{Lemma}
\begin{Remark}
The lemma implies that our sharp asymptotics for the critical value  coincides with the upper bound for the critical  value in (\ref{konno1}).
\end{Remark}
\begin{proof}{\bf\ of Lemma~\ref{thetaG}}
By the Markov property
\bea{asd1}\nn G(0,0)&=&1+\frac{1}{2d}\sum_{n=1}^\infty\pr{V_{n-1}\in\eth_d}\\\nn&=&1+\frac{1}{2d}\sum_{n=1}^\infty \pr{V_n\in\eth_d},\eea where the last line follows since $V_0=\bf{0}$.
\qed
\end{proof}
\medskip

Let us say a few words about the proofs. The structure of Konno's proofs of Theorem~\ref{KTheorem} follows the ideas of Bramson et alii
(see \cite{BDS}), who studied the long range contact process (LRCP)
in a limiting r\'{e}gime, when the range $M$ of the contact process, goes to infinity. The
set up of Konno is very similar to that of Bramson et alii in \cite{BDS}, with the difference that  the stirring speed
 goes to infinity, and not the range. Note that in \cite{BDS}, the authors were able to prove an asymptotic result for LRCP which was later improved by Durrett and Perkins (see \cite{DP}) where a sharp asymptotics for the convergence of the critical value was obtained for dimensions $d\geq2$. To prove the sharp asymptotics,  Durrett and Perkins, in~\cite{DP}, had further rescaled space and time
%to get into the Brownian r\'{e}gime
and proved weak convergence of  the rescaled processes to super-Brownian motion with drift. This convergence almost immediately gives  the lower bound for the critical value for LRCP. As for the upper bound,  Durrett and Perkins bounded the LRCP from below by an oriented percolation process (for this they used again
   convergence to super-Brownian motion), and, by this, the upper bound for the critical value was derived.\newline

As for the proof of our main result---Theorem~\ref{C3}---it follows immediately that from~(\ref{konno1}) and  Lemma~\ref{thetaG}, that
it is sufficient to prove just the lower bound for the critical value. This makes the proofs far less complicated than those in~\cite{DP}. In fact,  we prove our result without proving a weak convergence result of rescaled processes  to super-Brownian motion, which was one of the main technical ingredients of the proofs in~\cite{DP}.

The rest of the paper is organized as follows. Formal definitions of the contact process and "speeded-up" contact process are given in Section~\ref{sec:2}.
Theorem~\ref{C3} is proved in Section~\ref{sec:3}.

\section{Formal definitions}
\label{sec:2}

Before we proceed to the proofs, let us give formal definitions in this section.

The contact process with rapid stirring takes place on the lattice $\Z^d$.
Fix parameter $\theta$ for this process.
The
state of the process at time $t$ is given by a function $\xi_t^N:\Z^d\rightarrow\{0, 1\}$, where the value of $\xi_t^N(x)$ is determined by the number of particles present at $x$
at time $t$. Assume that $ \xi_0^N=\delta_0$. Independently of each other:

\begin{enumerate}
  \item particles die at rate $1$ without producing offspring;
  \item particles split into two at rate
 $1+\theta/N$. If split occurs at $x\in \Z^d$, then one of
the particles replaces the parent, while the other is sent to a site $y$ chosen according to a uniform distribution on $\eth_d(x)$ (the nearest neighbouring sites of $x$). If a newborn particle lands on an occupied site, its birth is suppressed;
  \item for each $x,y\in \Z^d$, with $x-y\in \eth_d$, the values of $\xi^N$ at $x$ and $y$ are
 exchanged at rate $N$ (stirring).
%particles jump to a randomly chosen neighbouring site at rate $2dN$. If the chosen site is occupied the particles exchange positions.
\newline

%\hspace{-0.7cm} An additional rule states:
%  \item If a newborn particle lands on an occupied site, its birth is suppressed.
\end{enumerate}
Just to clarify, when we say that events occur at a certain rate, we mean that times between events are independent exponential random variables with that rate. Let us also  make a comment about rule  3 above.
In terms of particles  dynamics, it means that whenever exchange between sites
$x$ and $y$ occurs, a particle at $x$ (if exists) jumps to $y$, and at the same time a particle at $y$ (if exists) jumps to $x$. If one follows the motion  of a typical particle, then, in the absence of branch events, it undergoes a symmetric random walk on $\Z^d$, with jumps at rate
 $2dN$.
\newline

%Our goal is to study
%\[\lambda_c(N)=\inf\{\lambda:\lim_{t\rightarrow\infty}\pr{\xi_t^N\not=0\given\xi_0^N=\delta_0}>0\},\]
%the critical value of survival, starting from a single particle at
%the origin. \newline

For our proofs it will be convenient  to deal with
the speeded-up contact
process $\hat\xi^N_t:\Z \rightarrow \N\cup\{0\}$.  This process is defined as $\hat\xi^N_t=\xi^N_{Nt}, t\geq 0$. Clearly,
$ \hat \xi_0^N=\delta_0$. Obviously this process obeys the same rules as $\xi^N$, just all the events occur with rate multiplied by $N$. In particular, particles die and split (if possible) with rates $N$ and $N+\theta$ respectively;
stirring between any two neighbouring sites occurs with rate $N^2$.

%Assume that independently of each other:

%\begin{enumerate}

%  \item particles die at rate $N$;
%  \item particles die and give split with rate $N+\theta$, where
%$\theta>0$. If a birth occurred at $x$, one of the offsprings
%   is sent to $y$ such that $\|x-y\|_1\equiv\sum_{i=1}^d|x_i-y_i|=1,$
%   while the other remains in its parent's place.;
%  \item particles jump at rate  $2dN^2$;
%  \item  If a split occurred and the newborn is sent to a site occupied by another particle,\footnote{LM I erased "... from the same lineage}
%coming from the same lineage at time 0,
%the birth is suppressed.
 %  \item If a particles jumps to $y$ such that $\|x-y\|_1=1$, if $y$ is vacant, otherwise the particles exchange places.
%\end{enumerate}
%and the above mechanisms are independent of each other.\newline

\vspace{0.5cm}

\section{Proof of Theorem \ref{C3}}
\label{sec:3}

As we have mentioned already, (\ref{konno1}) and Lemma~\ref{thetaG} imply that Theorem \ref{C3} will follow from the lower bound for $\lambda_c(N)-1$ obtained in the next proposition.
\begin{Proposition}\label{T2-1} For $d\geq 3$  \[\liminf_{N\rightarrow\infty}\frac{\lambda_c(N)-1}{\vartheta/N}\geq
1.\]
\end{Proposition}
%This section is devoted to the proof of this proposition.
In fact, Proposition~\ref{T2-1} follows easily from the following crucial result.
Recall that $\hat\xi^N_{t}=\xi^N_{Nt}, t\geq 0,$ was defined in Section~\ref{sec:2}.
\begin{Proposition}\label{L111} Fix an arbitrary $\theta<\vartheta$. Then there exist,  $N_\theta>0$ and $t_0=t_0(\theta)$, such that, for all $N>N_\theta$ and $t>t_0$,
\[\hat m_t^N \equiv  \e\left[\lb|\hat\xi_t^N\rb|\right]\leq e^{-\frac{1}{2}(\vartheta-\theta)t},\]

where $|\cdot|$ denote the total number of particles in the process.
\end{Proposition}

\begin{proof} {\textbf{of Proposition \ref{T2-1}}}\newline
Fix $\theta<\vartheta$, and choose
$N_\theta$ as in Proposition \ref{L111}. Then,
by this proposition, and the fact that the number of particles is an integer, we have that

\bea{asd2}\nn\pr{\lb|\hat\xi_t^N\rb|=0}&\geq&1-\hat m_t^N\\\nn&\geq&1-e^{-\frac{1}{2}(\vartheta-\theta)t}\\\nn&\rightarrow&1,\eea
as $t\rightarrow\infty$, for all $N\geq N_\theta$.\newline

From this, it follows immediately that for $N>N_{\theta}$,  $\{\hat\xi_t^N\}_{t\geq 0}$ dies out  in finite time, with probability one. The same happens with
$\lb\{\xi_t^N\rb\}_{t\geq 0}$ with probability one.\newline

Thus we have shown that for any $\theta<\vartheta$, there exists an
$N_\theta$ such that for every $N>N_\theta$,
\[\pr{\xi_t^N\not=\emptyset\mbox{ for all }t>0\given\xi_0^N=\delta_0}=0.\]

Therefore,
\[\inf\{\theta:\pr{\xi_t^N\not=\emptyset\mbox{ for all }t>0\given\xi_0^N=\delta_0}>0\}\geq\vartheta,\]

and, by definition of $\lambda_c(N)$, the proof of Proposition \ref{T2-1} is finished.\qed
\end{proof}
\vspace{0.5cm}
\begin{Remark}
By Lemma \ref{thetaG} and (\ref{konno1}), we have proven Theorem \ref{C3} modulo Proposition \ref{L111}.
\end{Remark}

The rest of this section is devoted to the proof of Proposition~\ref{L111}.
Before we proceed to its proof, let us first define some additional notation.
\begin{itemize}
\item Denote particles by Greek letters $\alpha, \beta, \gamma$ with the convention that $\alpha_0$ is the ancestor of $\alpha$ in generation 0. We use the branching process representing, so $\beta=(\beta_0,0,0,1,1,...)$ means that $\beta$ descends from $\beta_0$ through $(\beta_0, 0)$ then $(\beta_0,0,0)$ etc. For the rest of the paper, as we start from the single particle, we set $\beta_0=1$.
\item Consider that particles change names every time an event they branch (die or split).
\item Let $\alpha\wedge\beta$ denote the most recent common ancestor of $\alpha$ and $\beta$.
\item In the
 ``speeded-up''  process $\{\hat\xi_t^N\}_{t\geq 0}$, let $T_\alpha$ denote that time at which $\alpha$ branches (dies or splits) and let $B^\alpha_t$ be the location of $\alpha$ lineage at $t$ with the convention that $B^\alpha_t=\Delta$ if the particle is not alive at $t$.
\end{itemize}

Let
\bea{63} \tau_N=\frac{\ln N}{N^2}>0. \eea

The idea behind the proof is that in order to bound from above the total mass of the process, we can ignore collisions between distant relatives. To this end, we define a sequence of times $\tau_N$ such that collisions between relatives farther related than $\tau_N$, in the process $\hat \xi_t^N$, can be ignored. Roughly, two particles starting from the same position, need just above $N^{-2}$ units of time to "get lost", so, after that time they never meet again with very high probability. That is why we choose  $\tau_N$ as in $(\ref{63})$ --- $\tau_N$ is just a little bit larger than $N^{-2}$. A similar idea  is used in a number of papers (see e.g. \cite{DP} for the long range contact process).\newline

Let \[Z_1(t)=1(T_1\in[t,t+\tau_N), B^\beta_{t+\tau_N}-B^\gamma_{t+\tau_N}\in \eth_d),\] where $\beta=(1,0)$ and $\gamma=(1,1)$ are the children of $1\equiv \beta_0=\gamma_0$ -- the first particle. $Z_1(t)$ is the indicator of the event that the lineage of $1$ has exactly one splitting event in $[t, t+\tau_N)$, no deaths and its two offsprings ($\beta$ and $\gamma$) are alive and neighbours at time $t+\tau_N$. Note that the condition $B^\beta_{t+\tau_N}-B^\gamma_{t+\tau_N}\in \eth_d$ implies that there no deaths, as both particles must not be in $\Delta$ at time $t$. The same condition also implies that there were no more splits in $1$-th line,
since $\beta$ and $\gamma$ are the children of the particle $1$. Set $Z_1\equiv Z_1(0)$.\newline

Before we proceed to the actual proof of Proposition \ref{L111} we will need an  auxiliary result.
Let $\{V_t^N\}_{t\geq 0}$ be a continuous time, symmetric random walk on $\Z^d$ jumping with rate
$4dN^2$ and starting at the origin. Let $\{W_t^N\}_{t\geq 0}$ be a continuous time Markov chain taking
values in $\Z^d$, starting from the origin, and evolving as follows. If $W_t^N=x\in\eth_d$ then,
with rate $(4d-1)N^2$, $W^N_t$ makes a jump to $y\in\Z^d$, whereas with probability $\frac{4d-2}{4d-1}$ $y$
is chosen uniformly from
$\eth_d(x)\setminus\{\mathbf{0}\}$, and with probability $\frac{1}{4d-1}$, $y=-x$. If $W_t^N=x\not\in\eth_d$, then, with rate $4dN^2$, $W^N$ makes a jump to $y$ uniformly distributed in $\eth_d(x)$.\newline

Note that $W_t^N$ describes the behaviour of the difference in locations of two typical particles
in the process $\hat\xi^N$ in the absence of branching events. Such particles  move around independently like
symmetric random walks, with jumps rates $2dN^2$,   until they become neighbours. While they are neighbours, their behaviour is dictated by the stirring rules.

\begin{Lemma}\label{RWalks}
\[\ex{\int_0^t 1(V_s^N\in\eth_d)ds}=\ex{\int_0^t1(W_s^N\in\eth_d)ds}.\]
\end{Lemma}

\begin{proof}
Given $V_t^N\in\eth_d$, define $T_v$ to be the time $V^N$ spends in $\eth_d$ before leaving the set $\eth_d\cup \{\bf{0}\}$. Then clearly,
\[T_v=\sum_{i=1}^R\epsilon_i,\]

where $\epsilon_i$s are independent random variables distributed according to exponential distribution with rate $4dN^2$ and $R$ is independent of them and is geometric with parameter  $\frac{2d-1}{2d}$. Clearly $T_v$ is exponentially distributed with $\ex{T_v}=(2(2d-1)N^2)^{-1}$.\newline

Similarly, at every visit to $\eth_d$, $W_t^N$ spends in $\eth_d$ a time $T_w=\sum_{i=1}^{R'}\epsilon'_i$, where $R'$ is geometric with parameter $\frac{4d-2}{4d-1}$ and $\epsilon_i'$ are independent exponentially distributed random variables with rate $(4d-1)N^2$ and independent of $R'$. Thus, $T_w$ is exponentially distributed with the mean
\[\lb((4d-1)N^2\frac{4d-2}{4d-1}\rb)^{-1}=(2(2d-1)N^2)^{-1}.\]

Thus we may couple the processes together by setting them to be equal, every time they exit $\eth_d\cup\{\bf{0}\}$, and as this does not change time they spend in $\eth_d$ the result follows.\qed
\end{proof}

\begin{Lemma}\label{L661}
\[
\lim_{N\rightarrow\infty}N\ex{Z_1}=d\vartheta.\]
\end{Lemma}

\begin{proof}
First we calculate the probability of $F_1$, the event that there is exactly one birth in 1's lineage in $\tau_N$ units of time and no deaths on any of the branches. As the births in $\hat\xi^N$ process occur according to a Poisson process with rate $N+\theta$,  we have that

\bea{F1}\pr{F_1}&=&(N+\theta)\tau_Ne^{-(N+\theta)\tau_N}e^{-N\tau_N}\cdot\frac{1}{\tau_N}\int_0^{\tau_N}e^{-(2N+\theta)s}ds\\
\nn&=&\frac{N+\theta}{2N+\theta}e^{-(2N+\theta)\tau_N}\lb( 1-e^{-(2N+\theta)\tau_N} \rb).\eea

\[\ex{Z_1}=\pr{B^\beta_{\tau_N}-B^\gamma_{\tau_N}\in\eth_d\given F_1}\pr{F_1},\]

where $\beta$ and $\gamma$ are the offspring of $1$ alive at $\tau_N$.\newline

\bea{83}\hspace{-1cm}\pr{B^\beta_{\tau_N}-B^\gamma_{\tau_N}\in\eth_d \given
F_1}&=&\frac{1}{\tau_N}\int_0^{\tau_N}
\pr{W+W^N_{\tau_N-t}\in\eth_d}dt,\eea where $W$ is uniform on
$\eth_d$ and is the difference of positions of the two children of
$1$, right after the split; $\{W^N_t\}_{t\geq 0}$ is a continuous time Markov process defined in Lemma \ref{RWalks} independent of $W$.\newline

Change the variable in the integral, use Lemma \ref{RWalks} to get that
(\ref{83}) is equal to

\bea{84}&&\hspace{-1cm}\frac{1}{\tau_N}\int_0^{\tau_N}
\pr{W+V^N_s\in\eth_d}ds\\\nn &=&\frac{1}{\tau_N}\int_0^{\tau_N}
\sum_{n=0}^\infty\pr{W+V_n\in\eth_d}\frac{e^{-4dN^2s}(4dN^2s)^n}{n!}ds, \eea

where $V_n$ is a simple symmetric random walk on $\Z^d$ independent of $W$.\newline

Now let $\{\pi(u)\}_{u\geq 0}$ be a Poisson process with rate $1$ defined on the same probability space and independent of $W$ and $\{V_n\}_{n\geq 0}$. Define \[ h(u)\equiv\pr{W+V_{\pi(u)}\in\eth_d}.\]

Then, by independence of $\{\pi(u)\}_{u\geq 0}$ and $\{W+V_n\}_{n\geq 0}\,$ (\ref{84}) can be written as

\be{85} \frac{1}{\tau_N}\int_0^{\tau_N} h(4d N^2s) ds.
\ee

So, from (\ref{F1}), (\ref{85}), and (\ref{63}) we have that
\bea{86} \hspace{-2cm}\ex{Z_1}&=&\frac{N+\theta}{2N+\theta}e^{-2(N+\theta)\tau_N}\lb(1-e^{-(2N+\theta)\tau_N}\rb)\frac{1}{\tau_N}\int_0^{\tau_N}h(4d
N^2s)ds\\\nn&=&\frac{1}{4d
N^2\tau_N}\frac{N+\theta}{2N+\theta} e^{-(2N+\theta)\tau_N}\lb(1-e^{-(2N+\theta)\tau_N}\rb)\int_0^{4d
N^2\tau_N}h(r)dr,\eea where the
last equality follows from changing the variable inside the integral.\newline

Now let $N\rightarrow\infty$, use the Taylor expansion for the
exponential and the monotone convergence theorem to get that

\[ \lim_{N\rightarrow\infty}N\ex{Z_1}=\frac{1}{4d}\int_0^\infty\pr{W+V_{\pi(s)}\in\eth_d}ds. \]

The times between jumps of $V_{\pi(s)}$ are exponential with mean $1$.
Therefore \be{88}
\lim_{N\rightarrow\infty}N\ex{Z_1}=\frac{1}{4d}\sum_{n=1}^\infty\pr{V_n\in\eth_d}.\ee
This finishes the proof of this lemma. \qed
\end{proof}
\vspace{0.5cm}

\begin{proof}{\textbf{ of Proposition \ref{L111}}}\newline
Set $m_t^N=\ex{|\xi_t^N|}$. From (1.5) of \cite{Kon}, we have
\bea{asd3}\nn  m_{Nt}^N&=&1+\int_0^{Nt}\frac{\theta}{N} m_s^Nds-\frac{1}{2d}\int_0^{Nt}  I_s^Nds\\\nn&=&1+\int_0^t\theta  m_{Ns}^Nds-\frac{1}{2d}\int_0^t N  I_{Ns}^Nds,\eea

where $I_s^N$ is twice the expected number of pairs of neighbours in $\xi_{s}^N$  at time~$s$.
Let $\A(s)$ be the set of particles which are alive at $s$ in $\hat\xi_s^N$, and
 \[\hat I_s^N=\e\left[\sum_{\alpha,\beta\in \A(s)}1(B^\alpha_s-B^\beta_s\in \eth_d)\right]\] be twice the expected number of pairs of neighbours in $\hat \xi_t^N$ at time $s$.\newline

Thus, with $\hat m_t^N=m_{Nt}^N=\ex{|\xi_{Nt}^N|}$, we have
\be{Kon15}\hat m_t^N=1+\int_0^t\theta \hat m_s^Nds-\frac{1}{2d}\int_0^tN\hat I_s^Nds.\ee

Clearly, \[\hat m_t^N\leq 1+\theta\int_0^t \hat m_s^Nds,\spa \forall t\geq 0.\]

Therefore, as $1$ is a non-decreasing function, Gr\"{o}nwall's lemma gives that

\be{Gron1}\hat  m_t^N\leq e^{\theta(t-r)}\hat m_r^N,\spa \forall r\in[0,t].\ee

In particular,
\be{mstar}\hat m_{\tau_N}^N\leq e^{\theta\tau_N}\ee and we only need to take care of $t\geq\tau_N$.\newline

Note that (\ref{Kon15}) can be also written as
\be{Kon15-n}\hat m_t^N=\hat m_{\tau_N}^N+\int_{\tau_N}^t\theta \hat m_s^Nds-\frac{1}{2d}\int_{\tau_N}^tN\hat I_s^Nds.\ee

We extend the definition of $Z_1$ to all particles. To this end,  for any particle $\alpha$, set \[Z_\alpha(t)=1(T_\alpha\in[t,t+\tau_N), B^\beta_{t+\tau_N}-B^\gamma_{t+\tau_N}\in \eth_d),\] where $\beta=(\alpha, 0)$ and $\gamma=(\alpha,1)$ are the children of $\alpha$.
Further, let $\zeta_\alpha(t)$ be the indicator of the event that one of $\alpha$'s children created in $[t,t+\tau_N)$ died at the time of its birth $T_\alpha$ as a result of a collision with another particle. For the original particle $\alpha=\beta_0=1$, we have $\zeta_1(t)=0$, since there are no other particles around. For any $s\geq\tau_N$, let us now  give a lower bound for
\[ \ex{Z_\alpha(s-\tau_N)\given \F_{s-\tau_N}}.\]
Note that given $\F_{s-\tau_N}$, for $s\geq\tau_N$ and $\alpha \in \A(s-\tau_N)$, $Z_\alpha(s-\tau_N)$ is stochastically less than $Z_1$,
since one of  $\alpha$'s children created in $[t,t+\tau_N)$ may be killed  as a result of a collision at the time of its birth $T_\alpha$, which can not happen to $1$'s children due to lack of other particles. However, if we ``return'' the killed children back, then we easily get that, conditionally on
$\F_{s-\tau_N}$, for  $s\geq\tau_N$ and  $\alpha \in \A(s-\tau_N)$, $Z_\alpha(s-\tau_N)+\zeta_{\alpha}(s-\tau_N)$ is stochastically greater than $Z_1$.
%$T_{\alpha}$ is independent of the behaviour of the
%particles other than $\alpha$ on the time interval $[s-\tau_N, s]$. Moreover, conditionally on $D^c_{\alpha}(s-\tau_N)$,
% $Z_\alpha(s-\tau_N)$ depends just on the time of the split $T_{\alpha}$, but not on position of the particle before the split. Alltogether, this implies that  conditionally on $\F_{s-\tau_N}$ and  $\alpha \in \A(s-\tau_N)$,
%$  \ex{Z_\alpha(s-\tau_N)\given D^c_{\alpha}(s-\tau_N) \cap \F_{T_{\alpha}}} $ is equal in distribution to
% $\ex{Z_1\given \F_{T_{1}}} $, and moreover
%\begin{eqnarray}
%  \ex{\ex{Z_\alpha(s-\tau_N)\given D^c_{\alpha}(s-\tau_N) \cap \F_{T_{\alpha}}}\;\given \F_{s-\tau_N}} = \ex{Z_1}.
%\end{eqnarray}
Therefore we immediately get, that for $s\geq\tau_N$, $\alpha \in \A(s-\tau_N)$,
\begin{eqnarray*}
 \ex{Z_\alpha(s-\tau_N)\given \F_{s-\tau_N}}
 &\geq&  \ex{Z_1}- \ex{ \zeta_{\alpha}(s-\tau_N) \,\given \F_{s-\tau_N}}.
\end{eqnarray*}
%\newline
Use this to get, for $s\geq\tau_N$,
\bea{Ieq1}\hat I_s^N&\geq&\ex{\sum_{\alpha,\beta\in \A(s)\atop T_{\alpha\wedge\beta}>s-\tau_N}1(B^\alpha_s-B^\beta_s\in \eth_d)}\\\nn
&=&2 \ex{\sum_{\alpha\in\A(s-\tau_N)}\ex{Z_\alpha(s-\tau_N)\given \F_{s-\tau_N}}}\\\nn
&\geq&2\ex{\sum_{\alpha\in\A(s-\tau_N)}\left(\ex{Z_1}-\ex{\zeta_\alpha(s-\tau_N)\,\given \F_{s-\tau_N} }\right)}
%\\\nn&\geq&2\ex{\ex{\sum_{\alpha\in \A(s-\tau_N)}Z_\alpha(s-\tau_N)\given \F_{s-\tau_N}}}-2\e\sum_{\alpha\in\A(s-\tau_N)}\zeta_\alpha
\\\nn&=&2\ex{\hat\xi^N_{s-\tau_N}\ex{Z_1}}-2\ex{\sum_{\alpha\in\A(s-\tau_N)}\zeta_\alpha(s-\tau_N)}
\\\nn&=&2\ex{\hat m_{s-\tau_N}^N}\ex{Z_1}-2\ex{\sum_{\alpha\in\A(s-\tau_N)}\zeta_\alpha(s-\tau_N)}.
\eea
Recall that $\theta<\vartheta$ and set  $\varepsilon=(\vartheta-\theta)/2$. Then by Lemma \ref{L661}, we can choose $N_0>\theta$ sufficiently large
such that for any $N\geq N_0$,
\be{112}
N\ex{Z_1}\geq d\left(\vartheta-\frac{\varepsilon}{4}\right). \ee
Now use the bounds (\ref{112}), (\ref{Ieq1}),  $(\ref{mstar})$
to derive from (\ref{Kon15-n}) that

\be{Ieq2}\hat m_t^N\leq e^{\theta\tau_N}+\theta\int_{\tau_N}^t \hat m_s^Nds-\lb(\vartheta-\frac{\varepsilon}{4}\rb)\int_{\tau_N}^t\hat m_{(s-\tau_N)}^Nds+\frac{1}{d}\int_{\tau_N}^t\Xi_s^Nds,\; \forall N\geq N_0,\ee

where $\Xi_s^N=N\ex{\sum_{\alpha\in\A(s-\tau_N)}\zeta_\alpha(s-\tau_N)}$.

On the other hand for $N\geq N_0>\theta$,

\bea{Xi1}\Xi_s^N&\leq&N \e\left[ \sum_{\alpha\in\A(s-\tau_N)}\p\left(T_\alpha\in[s-\tau_N,s) \,\given \F_{(s-\tau_N)}\right)
\right.
\\\nn&&\left. \times
\frac{1}{2d\tau_N}\ex{\int_{s-\tau_N}^s\sum_{\gamma\in\A(r)}1(B^\alpha_r-B^\gamma_r\in \eth_d)dr
\,\given \F_{(s-\tau_N)}}\right]\\\nn&=&
\frac{N}{2d\tau_N}(2N+\theta)\tau_Ne^{-(2N+\theta)\tau_N}\ex{\int_{s-\tau_N}^s\sum_{\alpha,\gamma\in \A(r)\atop }1(B^\alpha_r-B^\gamma_r\in \eth_d)dr}\\\nn&\leq&\frac{2}{d}N^2\int_{s-\tau_N}^s \hat I_r^Ndr.\eea

Changing the order of integration yields
\bea{Xi2}\int_{\tau_N}^t\int_{(s-\tau_N)_+}^s \hat I_r^Ndrds&\leq&\int_0^t\int_r^{(r+\tau_N)\wedge t} \hat I_r^Ndsdr\\\nn&\leq&\tau_N\int_0^t \hat I_r^N dr\\\nn&\leq&2dN^{-1}\tau_N\lb(1+\int_0^t\theta \hat m_s^Nds\rb),\eea

where the last inequality follows from (\ref{Kon15}).\newline

From (\ref{Xi1}) and (\ref{Xi2}) we get that \[\frac{1}{d}\int_{\tau_N}^t \Xi_s^Nds\leq \frac{4}{d}N\tau_N\lb(1+\int_0^t \hat m_s^Nds\rb),\; \forall N\geq N_0\,.\]

Thus, (\ref{Ieq2}) can be rewritten as
\bea{Ieq3}\hat m_t^N&\leq& e^{10\theta N\tau_N} + \theta\lb(1+\frac{4}{d}N\tau_N\rb)\int_0^t\hat m_s^Nds
\\\nn &&\mbox{}-\lb(\vartheta-\frac{\varepsilon}{4}\rb)\int_{\tau_N}^t\hat m_{s-\tau_N}^Nds,\;\forall N\geq N_0\,,\eea

where without loss of generality we assumed that $N_0$ is sufficiently large so that $e^{10\theta N\tau_N}\geq e^{\theta\tau_N}+\frac{4}{d}N\tau_N\,,$ for any $N\geq N_0$.
Now use  (\ref{Gron1}) to get
%The latter also immediately yields,
%Thus,
\[\hat m_{t-\tau_N}^N\geq \hat m_t^Ne^{-\theta\tau_N}, \spa \forall t\geq 0.\]
Choose $N_1\geq N_0$ large enough so that for any $N\geq N_1$,
\bea{}\nn
(\vartheta-\varepsilon/2)&\leq& (\vartheta-\varepsilon/4)e^{-\theta\tau_N},
\\\nn
 e^{10\theta N\tau_N}+(\vartheta-\varepsilon/2)\tau_Ne^{\theta\tau_N}&\leq&e^{1}.
\eea

 Then we get,

\bea{asd4}\nn \hat m_t^N&\leq& e^{10\theta N\tau_N} + \theta\lb(1+\frac{4}{d}N\tau_N\rb)\int_0^t\hat m_s^Nds-(\vartheta-\varepsilon/2)\int_{\tau_N}^t\hat m_s^Nds\\\nn&=&e^{10 N\theta\tau_N} + (\theta-\vartheta+\varepsilon/2+4d^{-1}\theta N\tau_N)\int_0^t\hat m_s^Nds + (\vartheta-\varepsilon/2)\tau_Ne^{\theta\tau_N}
\\\nn&\leq&
%e^{20\theta N\tau_N}
e^1+ (\theta-\vartheta+\varepsilon/2+4d^{-1}\theta N\tau_N)\int_0^t\hat m_s^Nds.\eea

Again, by Gr\"{o}nwall's lemma,
\be{111}
\hat m_t^N\leq e^{(\theta-\vartheta+\varepsilon/2+2\theta N\tau_N)t+1},\spa \forall t>0, N\geq N_1.
\ee

Now choose $N_\theta>N_1$ such that \[2\theta N\tau_N\leq \frac{\vartheta-\theta}{10}, \;\forall N\geq N_{\theta}\,.\]

Then recall that  $\varepsilon=(\vartheta-\theta)/2$ to get from~(\ref{111}) that  \[\hat m_t^N \leq e^{-\frac{13}{20}(\vartheta-\theta)t+1}.\]

Now we choose $t_0>0$, such that \[e^{-\frac{3}{20}(\vartheta-\theta)t_0+1}\leq 1,\]
and hence \[\hat m_t^N\leq e^{-\frac{1}{2}(\vartheta-\theta)t},\spa \forall N\geq N_\theta\,, t>t_0.\]\qed

\end{proof}

\paragraph{Acknowledgements.}
 Both %\footnote{LM acknowledgements added}
authors thank an anonymous referee for the careful reading of the manuscript, and for a number of
useful comments and suggestions that improved the exposition.
\newpage
\addcontentsline{toc}{section}{References}
\bibliographystyle{plain}
\bibliography{Reference}

\end{document}